\documentclass[12pt]{article}
\usepackage{amsmath, amssymb, amscd, amsfonts, tikz, tikz-cd, amsthm, undertilde, url}
\usepackage[left=1.0in, right=1.0in, top=1.0in, bottom=1.0in]{geometry}
\title{Ramification in the Inverse Galois Problem}
\author{Benjamin Pollak}
\date{}

\theoremstyle{plain}
\newtheorem{theorem}{Theorem}[section]
\theoremstyle{plain}
\newtheorem{corollary}[theorem]{Corollary}
\theoremstyle{definition}

\theoremstyle{definition}
\newtheorem{example}[theorem]{Example}
\theoremstyle{plain}
\newtheorem{lemma}[theorem]{Lemma}
\theoremstyle{remark}
\newtheorem{remark}[theorem]{Remark}
\theoremstyle{plain}
\newtheorem{proposition}[theorem]{Proposition}
\theoremstyle{plain}
\newtheorem{conjecture}[theorem]{Conjecture}

\begin{document}
\maketitle

\begin{abstract}
This paper focuses on a refinement of the inverse Galois problem. We explore what finite groups appear as the Galois group of an extension of the rational numbers in which only a predetermined set of primes may ramify. After presenting new results regarding extensions in which only a single finite prime ramifies, we move on to studying the more complex situation in which multiple primes from a finite set of arbitrary size may ramify. We then continue by examining a conjecture of Harbater that the minimal number of generators of the Galois group of a tame, Galois extension of the rational numbers is bounded above by the sum of a constant and the logarithm of the product of the ramified primes. We prove the validity of Harbater's conjecture in a number of cases, including the situation where we restrict our attention to finite groups containing a nilpotent subgroup of index $1,2,$ or $3$. We also derive some consequences that are implied by the truth of this conjecture. 
\end{abstract}

\section{Introduction}
This paper concentrates on how the set of ramified primes in a Galois extension of the rational numbers relates to the Galois group. For a square-free natural number $n \in \mathbb{N}$, define $U_n$ to be $\operatorname{Spec}\left(\mathbb{Z}\left[\frac{1}{n}\right]\right)$, an open subset of $\operatorname{Spec}\left(\mathbb{Z}\right)$. We denote the \'etale fundamental group by $\pi_1\left(U_n\right)$; it is the Galois group of the maximal extension of $\mathbb{Q}$ that is unramified at finite primes not dividing $n$. We then let $\pi_A(U_n)$ be the set of finite quotients of $\pi_1(U_n)$; it is the set of finite groups appearing as Galois groups of extensions of $\mathbb{Q}$ unramified at finite primes not dividing $n$. Finally, let $\pi_A^t(U_n)$ be the set of groups appearing when we restrict our attention to tame extensions. Our goal is to gain some insight into the contents of $\pi_A(U_n)$ and $\pi_A^t(U_n)$ for various choices of $n$. In studying the relationship between the ramification of Galois extensions and the Galois group, we obtain a refinement of the traditional inverse Galois problem which simply asks whether for every finite group $G$, does there exist an $n$ such that $G \in \pi_A(U_n)$.

For a finite group $G$, let
\[	d(G) = \min \left\lbrace \lvert S \rvert \; | S \text{ is a generating set for }G \right\rbrace. \]
In the function field case, a square-free polynomial $f \in \mathbb{F}_p[t]$ of degree $d$ has norm $p^d$. Let $U \subseteq \mathbb{A}^1_{\mathbb{F}_p}$ be the complement of the vanishing set of $f$ and $\pi_A^\text{t,reg}(U)$ be the finite groups appearing as Galois groups of tame, regular extensions of $\mathbb{F}_p(t)$ unramified outside of primes dividing $f$. Then, any $G \in \pi_A^\text{t,reg}(U)$ satisfies $d(G) \leq d = \log_p\left(\operatorname{Norm}(f)\right)$. Inspired by the analogy between number fields and function fields, in the arithmetic situation we view $U_n \subseteq \operatorname{Spec}\left(\mathbb{Z}\right)$ as the complement of the vanishing set of a square-free natural number $n$ that has norm $n$. In \cite{DH}, Harbater then proposes the following conjecture:
\begin{conjecture} \label{HConj}
	There is a constant $C$ such that for every square-free $n \in \mathbb{N}$, every $G \in \pi_A^t(U_n)$ satisfies $d(G) \leq \log\left(n\right) + C$.
\end{conjecture}

Since at least one finite prime ramifies in every extension of $\mathbb{Q}$, $\pi_A(U_1)$ consists of only the trivial group. Apart from this case, there is no other square-free $n$ for which $\pi_A(U_n)$ is completely understood. Nevertheless, we can obtain a partial description. Given a specific square-free $n$, the focus of Section 2 is to say as much as possible about extensions of $\mathbb{Q}$ unramified at finite primes not dividing $n$. We show in Theorem \ref{TRTLEQ53} that the Galois group of a tame, totally real, Galois extension ramified at a single prime of size at most $53$ is cyclic. We also show in Theorem \ref{GLT660} that if a group of order less than $660$ is the Galois group of an extension ramified at a single finite prime of size less than $37$, then the group must be solvable. Section $3$ is devoted to studying how generating sets of a Galois group relate to the ramified primes in the corresponding extension. In particular, we will prove in Proposition \ref{HCTN}, Theorem \ref{NI2}, and Theorem \ref{NI3} that if we restrict out attention to groups with a nilpotent subgroup of index $1,2,$ or $3$, then Conjecture \ref{HConj} is true. We conclude by examining some consequences that are implied by the truth of Conjecture \ref{HConj}.

\section{Galois Extensions of $\mathbb{Q}$ with Specified Ramification}

\subsection{Extensions Ramified at a Single Prime}
We will first focus on extensions of the rational numbers in which only one finite prime ramifies. We show in Theorem \ref{TRTLEQ53} that a tame, totally real, Galois extension ramified at a single finite prime of size at most $53$ is cyclic. In Corollary \ref{5SOLV} and Proposition \ref{5Arbitrary} we present results about extensions of $\mathbb{Q}$ in which $5$ is the only ramified finite prime. We then prove in Theorem \ref{GLT660} that Galois extensions of small degree that are ramified at a single, small finite prime are solvable.
\subsubsection{Totally Real Extensions}
Harbater proves in \cite{DH} that for $p < 23$ a prime number, the cyclotomic extension $\mathbb{Q}(\zeta_p)$ is the maximal extension of $\mathbb{Q}$ that is tamely ramified only at $p$ and $\infty$. We now present some analogous results in the totally real case in which the infinite place is also restricted from ramifying. For a square-free $n \in \mathbb{N}$, we let $\pi_1^{\operatorname{t,tr}}(U_n)^{\operatorname{solv}}$ denote the set of solvable groups that appear as the Galois group of some tame, totally real extension of $\mathbb{Q}$ in which only primes dividing $n$ may ramify.

\begin{proposition} \label{TRTSOLV}
	Let $p$ be an odd prime number. If $\mathbb{Q}(\zeta_p + \zeta_p^{-1})$ has class number $1$, then $\pi_1^{\operatorname{t,tr}}(U_p)^{\operatorname{solv}}$ is cyclic of order $\frac{p-1}{2}$. Hence, if $K/\mathbb{Q}$ is a totally real, tame, solvable extension only ramified at $p$, then $K \leq \mathbb{Q}(\zeta_p + \zeta_p^{-1})$.
\end{proposition}
\begin{proof}
	Let $p$ be a prime such that $\mathbb{Q}(\zeta_p + \zeta_p^{-1})$ has class number $1$. Suppose $G \in \pi_A^{\operatorname{t,tr}}(U_p)^{\operatorname{solv}}$. Let $K/\mathbb{Q}$ be an extension providing witness to the fact that $G \in \pi_A^{\operatorname{t,tr}}(U_p)^{\operatorname{solv}}$. Let $G^{(1)} = [G,G]$ and $G^{(2)} = [G^{(1)}, G^{(1)}]$ be the first and second commutator subgroups respectively. Letting $K^{G^{(1)}}$ and $K^{G^{(2)}}$ denote the fixed fields, we obtain the following diagram:
	$$\begin{tikzpicture}
	\matrix(m)[matrix of math nodes,
	row sep=3em, column sep=4.5em,
	text height=1.5ex, text depth=0.25ex]
	{ K	  \\
		K^{G^{(2)}} \\
		K^{G^{(1)}}	  \\
		\mathbb{Q} \\ };
	\path[-,font=\scriptsize]
	(m-1-1) edge node[auto] {} (m-2-1)
	(m-2-1) edge node[auto] {} (m-3-1)
	(m-3-1) edge node[auto] {} (m-4-1)
	;
	\end{tikzpicture}.$$
	Since $K^{G^{(1)}}$ is an abelian extension of $\mathbb{Q}$ that is totally real and tamely ramified only at $p$, by the Kronecker-Weber theorem $K^{G^{(1)}} \leq \mathbb{Q}(\zeta_p + \zeta_p^{-1})$ and $K^{G^{(1)}}/\mathbb{Q}$ is totally ramified. Note now that $K^{G^{(1)}}$ must have class number $1$. If not, it would have a nontrivial, unramified, abelian extension. However, taking the compositum of such an extension with $\mathbb{Q}(\zeta_p + \zeta_p^{-1})$ would then yield a nontrivial, unramified, abelian extension of $\mathbb{Q}(\zeta_p + \zeta_p^{-1})$, contradicting $\mathbb{Q}(\zeta_p + \zeta_p^{-1})$ having class number $1$. Thus, the abelian extension $K^{G^{(2)}}/K^{G^{(1)}}$ has no nontrivial, unramified subextensions, and so must be totally ramified. This implies that $K^{G^{(2)}}/\mathbb{Q}$ is totally ramified. By the tameness assumption, it must also be cyclic. Hence, $G/G^{(2)}$ is abelian, and so $G^{(1)} = G^{(2)}$. By assumption of $G$ being solvable, we conclude that $G^{(1)}$ must be trivial and $K^{G^{(1)}} = K$. Thus, $K \leq \mathbb{Q}(\zeta_p + \zeta_p^{-1})$ and $G$ is cyclic of order dividing $\frac{p-1}{2}$. 
\end{proof}

\begin{corollary} \label{TRTSOLV151}
	Suppose $p \leq 151$ is an odd prime. Then $\pi_1^{\operatorname{t,tr}}(U_p)^{\operatorname{solv}}$ is cyclic of order $\frac{p-1}{2}$; the maximal tame, totally real, solvable extension of $\mathbb{Q}$ ramified only at $p$ is  $\mathbb{Q}(\zeta_p + \zeta_p^{-1})$.
\end{corollary}
\begin{proof}
	By Theorem 1.1 in \cite{JM}, the class number of $\mathbb{Q}(\zeta_p + \zeta_p^{-1})$ is $1$ for $p \leq 151$. Now apply Proposition \ref{TRTSOLV}.
\end{proof}

\begin{remark}
	If the class number of $\mathbb{Q}(\zeta_p + \zeta_p^{-1})$ is larger than $1$, then the Hilbert class field of  $\mathbb{Q}(\zeta_p + \zeta_p^{-1})$ shows that $\pi_1^{\operatorname{t,tr}}(U_p)^{\operatorname{solv}}$ is not cyclic.
\end{remark}

For $p\leq 53$ we can drop the solvable assumption in Proposition \ref{TRTSOLV}.
\begin{theorem} \label{TRTLEQ53}
	Suppose $p \leq 53$ is an odd prime. Then $\pi_1^{\operatorname{t,tr}}(U_p)$ is cyclic of order $\frac{p-1}{2}$ and $\mathbb{Q}(\zeta_p + \zeta_p^{-1})$ is the maximal totally real, tame extension of $\mathbb{Q}$ that is ramified only at $p$.
\end{theorem}
\begin{proof}
	It suffices to prove the claim that $\mathbb{Q}(\zeta_p + \zeta_p^{-1})$ is the maximal totally real, tame extension of $\mathbb{Q}$ that is ramified only at $p$. In doing so, we need only consider Galois extensions; a non-Galois counterexample would provide a Galois counterexample by taking the Galois closure. So, suppose for contradiction that the claim is false. Let $K$ be a Galois extension of $\mathbb{Q}$ of minimal degree that contradicts it. By Corollary \ref{TRTSOLV151}, $G = \operatorname{Gal}\left(K/\mathbb{Q}\right)$ is non-solvable. Let $e$ denote the ramification index of the primes above $p$. Since the extension is tame and $p \leq 53$, by \cite[Chapter 3, Section 6]{JPS} the root discriminant of $K/\mathbb{Q}$ is at most 
	\[ p^{1 + v_p(e) - \frac{1}{e}} = p^{1 + 0 - \frac{1}{e}}  = p^{1- \frac{1}{e}} \leq 53^{1 - \frac{1}{e}} < 53. \]
	By \cite{FDyD}, any totally real extension of $\mathbb{Q}$ of degree $500$ or larger has root discriminant bigger than $53$. Hence, $[K:\mathbb{Q}] < 500$. By the minimality of $[K:\mathbb{Q}]$, every proper quotient of $G$ must be solvable. By Corollary \ref{TRTSOLV151}, every proper quotient is therefore abelian. By Lemma 2.5 in $\cite{DH}$, we conclude that $e \leq 14$. Thus, the root discriminant is at most
	\[ 53^{1-\frac{1}{14}} < 40.	\]
	By \cite{FDyD} again, we now get $[K:\mathbb{Q}] \leq 84$. The only non-solvable group of order at most $84$ is $A_5$. Thus, $G \cong A_5$. Once more by Lemma 2.5 in \cite{DH}, $e \leq 5$ and so the root discriminant is at most
	\[	53^{1- \frac{1}{5}} < 24.	\]
	Finally, \cite{FDyD} tells us that the root discriminant must be at least $36$ for degree $60$ totally real extensions of $\mathbb{Q}$. This is a contradiction.
\end{proof}

\subsubsection{Extensions Ramified at a Small Prime}
We now focus on extensions of $\mathbb{Q}$ ramified at a single, small, integral prime. We begin by adapting a result of Hoelscher in \cite{JLH} to more suitably apply to our needs.
\begin{proposition} \label{SOSPSOLV}
	Suppose $K/\mathbb{Q}$ is a nontrivial, solvable Galois extension ramified only at a single, odd finite prime, $p$, and possibly $\infty$. Let $G = \operatorname{Gal}\left( K/\mathbb{Q} \right)$. Then, either $G$ is a cyclic $p$-group, $G/p(G)$ is isomorphic to a nontrivial subgroup of $\mathbb{Z}/(p-1)\mathbb{Z}$, or $G$ has a cyclic quotient of order $p^t$ where $\mathbb{Q}(\zeta_{p^{t+1}})$ is the first $p$-power cyclotomic field with nontrivial class group.
\end{proposition}
\begin{proof}
	Let $K$ and $G$ satisfy the hypotheses above. Let $K_0/\mathbb{Q}$ be the maximal $p$-power, Galois sub-extension of $K/\mathbb{Q}$ and set $N = \operatorname{Gal}(K/K_0)$. By Theorem 2.11 in \cite{DH}, $N$ is cyclic. So, 
	\[ \text{Gal} \left( K_0 / \mathbb{Q} \right) \cong \mathbb{Z} / p^n \mathbb{Z} \text{ for some } n \geq 0. \]
	Furthermore, by Kronecker-Weber, $K_0$ is the cyclic sub-extension of degree $p^n$ in $\mathbb{Q}(\zeta_{p^{n+1}})$.
	$$
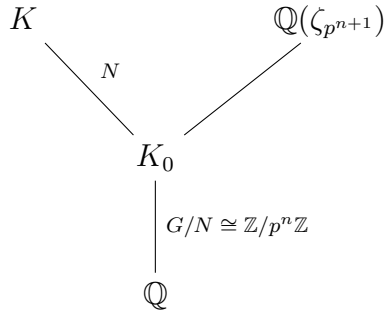
\begin{tikzpicture}
	\matrix(m)[matrix of math nodes,
	row sep=3em, column sep=2.5em,
	text height=1.5ex, text depth=0.25ex]
	{ K & & \mathbb{Q}(\zeta_{p^{n+1}}) \\
		& K_0 & \\
		&\mathbb{Q} &  \\};
	\path[-,font=\scriptsize]
	(m-1-1) edge node[auto] {$N$} (m-2-2)
	(m-2-2) edge node[auto] {$G/N \cong \mathbb{Z} / p^n \mathbb{Z} $} (m-3-2)
	(m-1-3) edge node[auto] {} (m-2-2)
	;
	\end{tikzpicture}.$$
	Suppose now that $G$ is not a cyclic $p$-group and that $G/p(G)$ is not isomorphic to a nontrivial subgroup of $\mathbb{Z}/(p-1)\mathbb{Z}$. We must show that $\mathbb{Q}(\zeta_{p^{n+1}})$ has nontrivial class group; this then shows that $n \geq t$ and so $G$ has a cyclic quotient of order $p^t$. We first show that $N/p(N)$ is not isomorphic to a nontrivial subgroup of $\mathbb{Z}/(p-1)\mathbb{Z}$.
	
	Seeking a contradiction, suppose that $N/p(N)$ is isomorphic to a nontrivial subgroup of $\mathbb{Z}/(p-1)\mathbb{Z}$. Then,
	\[ N / p \left( N \right) \cong \mathbb{Z} / m \mathbb{Z} \text{ for some } m>1 \text{ dividing } p-1. \]
	Letting $F$ denote the fixed field of $K$ under $p(N)$, we obtain the following diagram:
	$$
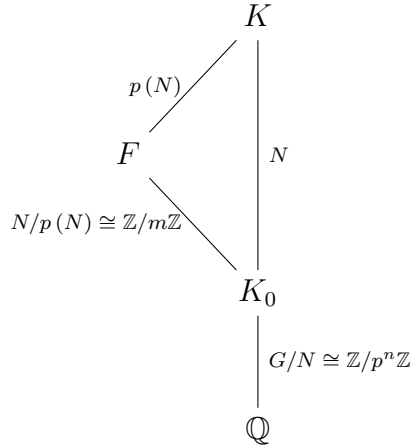
\begin{tikzpicture}
	\matrix(m)[matrix of math nodes,
	row sep=3em, column sep=2.5em,
	text height=1.5ex, text depth=0.25ex]
	{   & K \\
		F &  \\
		& K_0  \\
		&  \mathbb{Q} \\};
	\path[-,font=\scriptsize]
	(m-1-2) edge node[left] {$p \left( N \right)$} (m-2-1)
	(m-1-2) edge node[auto] {$N$} (m-3-2)
	(m-2-1) edge node[left] {$ N / p \left( N \right) \cong \mathbb{Z} / m \mathbb{Z}$} (m-3-2)
	(m-3-2) edge node[auto] {$G/N \cong \mathbb{Z} / p^n \mathbb{Z}$} (m-4-2)
	;
	\end{tikzpicture}.$$
	Since $N$ is normal in $G$ and $p \left( N \right)$ is characteristic in $N$, $p \left( N \right)$ is also normal in $G$. Hence, $F / \mathbb{Q}$ is a Galois extension with
	\[ \text{Gal} \left( F / \mathbb{Q} \right) \cong G / p \left( N \right).    \]
	Because $m \mid p-1$ and $\gcd(p-1,p) = 1$, the Schur-Zassenhaus Theorem tells us that 
	\[ \text{Gal} \left( F / \mathbb{Q} \right) \cong G / p \left( N \right) \cong \mathbb{Z}/ m\mathbb{Z} \rtimes  \mathbb{Z} / p^n \mathbb{Z}.	\]
	However, the automorphism group $\operatorname{Aut}(\mathbb{Z}/m\mathbb{Z})$ has order $\phi(m)$ which is prime to $p$. Thus, there are no nontrivial homomorphisms from $\mathbb{Z}/p^n\mathbb{Z}$ to $\operatorname{Aut}(\mathbb{Z}/m\mathbb{Z})$, and the above semidirect product is in fact a direct product. We conclude that
	\[ \text{Gal} \left( F / \mathbb{Q} \right) \cong G / p \left( N \right) \cong \mathbb{Z}/ m\mathbb{Z} \times  \mathbb{Z} / p^n \mathbb{Z}.	\]
	Noting that $ p \left( G \right) / p \left( N \right) \cong p \left( G/ p \left( N \right) \right) $ and applying the third isomorphism theorem, we obtain 
	\[ G / p \left( G \right) \cong \left( G / p \left( N \right) \right) / \left( p \left( G \right) / p \left( N \right) \right) \cong  \left( G / p \left( N \right) \right) /  p \left( G/ p \left( N \right) \right)  \]
	\[ \cong  \left( \mathbb{Z}/ m\mathbb{Z} \times  \mathbb{Z} / p^n \mathbb{Z} \right) / \left( \mathbb{Z} / p^n \mathbb{Z} \right) \cong \mathbb{Z} / m \mathbb{Z}.  \]
	This contradicts our assumption that $G / p \left( G \right)$ is not isomorphic to a nontrivial subgroup of $\mathbb{Z} / (p-1) \mathbb{Z}$. We conclude that $N/p(N)$ is not isomorphic to a nontrivial subgroup of $\mathbb{Z}/(p-1)\mathbb{Z}$.

	By Theorem 1.1 and Lemma 1.4 of \cite{JLH}, there is a nontrivial, abelian, unramified sub-extension $L / K_0 \left( \zeta_p \right)$ of $ K \left( \zeta_p \right) / K_0 \left( \zeta_p \right)$ of degree prime to $p$ with $L$ Galois over $\mathbb{Q}$:
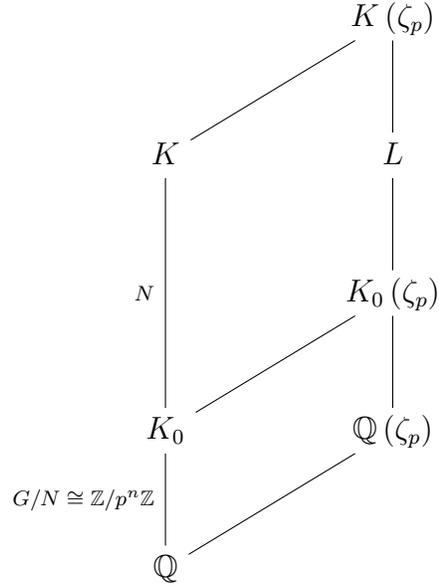
	$$\begin{tikzpicture}
	\matrix(m)[matrix of math nodes,
	row sep=3em, column sep=4.5em,
	text height=1.5ex, text depth=0.25ex]
	{ 	& K \left( \zeta_p \right) \\
		K	& L \\
		& K_0 \left( \zeta_p \right) \\
		K_0 & \mathbb{Q}\left( \zeta_p \right) \\  
		\mathbb{Q} & \\ };
	\path[-,font=\scriptsize]
	(m-2-1) edge node[auto] {} (m-1-2)
	(m-2-2) edge node[auto] {} (m-1-2)
	(m-3-2) edge node[auto] {} (m-2-2)
	(m-4-1) edge node[auto] {$N$} (m-2-1)
	(m-4-1) edge node[auto] {} (m-3-2)
	(m-5-1) edge node[auto] {$ G/N \cong \mathbb{Z} / p^n \mathbb{Z} $} (m-4-1)
	(m-5-1) edge node[auto] {} (m-4-2)
	(m-4-2) edge node[auto] {} (m-3-2)
	;
	\end{tikzpicture}.$$
	Since $K_0 \leq \mathbb{Q}(\zeta_{p^{n+1}})$ and $[K_0:\mathbb{Q}] = p^n$, it must be the case that $K_0(\zeta_p) = \mathbb{Q}(\zeta_{p^{n+1}})$. Since $L$ is a nontrivial abelian unramified extension of $\mathbb{Q} \left( \zeta_{p^{n+1}} \right)$, the class number of $\mathbb{Q} \left( \zeta_{p^{n+1}} \right)$ is not $1$.
\end{proof}

\begin{corollary}
	Let $p < 23$ be an odd prime and let $K/\mathbb{Q}$ be a nontrivial, solvable Galois extension ramified only at $p$ and possibly $\infty$ with $G = \operatorname{Gal}(K/\mathbb{Q})$. One of the following holds:
	\begin{enumerate}
		\item $G/p(G)$ is a nontrivial subgroup of $\mathbb{Z}/(p-1)\mathbb{Z}$.
		\item $G$ has a cyclic quotient of order $p$.
	\end{enumerate}
\end{corollary}
\begin{proof}
	Apply Proposition \ref{SOSPSOLV} while noting that the $p^\text{th}$ cyclotomic field has class number $1$ for $p < 23$ a prime.
\end{proof}

In \cite{DH}, Harbater obtains results about extensions of the rational numbers in which the only finite prime that ramifies is $2$. In $\cite{JLH}$, Hoelscher studies extensions in which the only finite prime that ramifies is $3$. We now turn our attention to extensions in which the only finite prime that ramifies is $5$.
\begin{corollary} \label{5SOLV}
	Let $G$ be the Galois group of a nontrivial, solvable extension ramified only at $5$ and possibly $\infty$. One of the following holds:
	\begin{enumerate}
		\item $G$ is a cyclic $5$-group.
		\item $G / p \left( G \right) \cong \mathbb{Z} / 2\mathbb{Z}$.
		\item $G / p \left( G\right) \cong \mathbb{Z}/4 \mathbb{Z}$.
		\item $G$ has a cyclic quotient of order $25$.
	\end{enumerate} 
\end{corollary}
\begin{proof}
	The $125^\text{th}$ cyclotomic field is the first $5$-power cyclotomic field with nontrivial class group. Apply Proposition \ref{SOSPSOLV}.
\end{proof}

We now drop the solvable assumption and consider arbitrary extensions of $\mathbb{Q}$ in which only $5$ and $\infty$ may ramify.
\begin{proposition} \label{5Arbitrary}
	If $K/\mathbb{Q}$ is a nontrivial, Galois extension ramified only at $5$ and possibly $\infty$ with Galois group $G$, then one of the following holds:
	\begin{enumerate}
		\item $G \cong \mathbb{Z}/5\mathbb{Z}$.
		\item $G/p \left( G \right) \cong \mathbb{Z}/4 \mathbb{Z}$.
		\item $G/p \left( G \right) \cong \mathbb{Z} /2 \mathbb{Z}$.
		\item $e \equiv 0 \pmod 5$ and $e \geq 10$ , where $e$ is the ramification index of the primes above $5$.
	\end{enumerate}
\end{proposition}
\begin{proof}
	If $G$ is solvable, then one of the conditions in Corollary \ref{5SOLV} holds. If the second or third condition holds, then we are done. If the fourth condition holds, then, by Kronecker-Weber, the cyclic quotient of order $25$ produces a totally ramified sub-extension and so $25 \mid e$. Finally, if the first condition holds, either $G \cong \mathbb{Z}/5\mathbb{Z}$ and we are done, or $G \cong \mathbb{Z}/5^l\mathbb{Z}$ for some $l \geq 2$. Again by Kronecker-Weber, the $\mathbb{Z}/5^l\mathbb{Z}$ extension is totally ramified and so $25 \mid e$.

	Suppose now that $G$ is not solvable. Then, by the proposition in \cite{JLH}, $\lvert G \rvert > 300$. Let $n = \lvert G \rvert$ be the degree of the corresponding extension, and let $\Delta$ be the discriminant. Since the degree of the extension is at least 300, \cite{FDyD} tells us that $\lvert \Delta \rvert^{\frac{1}{n}} \geq 19.2$. Since $5$ is the only finite prime that ramifies, we have from \cite{JPS} that $\lvert \Delta \rvert^{\frac{1}{n}} \leq 5^{1 + v_5(e) - \frac{1}{e}}$. Thus,
	\[ 19.2 \leq 5^{1 + v_5(e) - \frac{1}{e}}. \]
	It must be the case that $v_5(e) > 0$, for otherwise the right hand side is at most $5$. Hence, $e \equiv 0 \pmod 5$. If $v_5(e) = 1$, then $e$ still cannot be $5$; if it were, the right hand side above is at most $18.12$. Thus, $e \geq 10$.
\end{proof}

\begin{remark}
	The fourth condition in Proposition \ref{5Arbitrary} can be replaced by $25 \mid e$ if one is willing to assume the generalized Riemann hypothesis. The proof showed that in the solvable case we can unconditionally replace the fourth condition with $25 \mid e$. By Theorem \ref{GLT660}, if $G$ is non-solvable we actually have $\left\vert{G}\right\vert \geq 660$. Under assumption of the generalized Riemann hypothesis, we have from Table 1 in \cite{AOT} that the Odlyzko lower bound on the root discriminant for fields of degree at least $340$ is $25.09$. This forces $v_5(e) \geq 2$ and so $25 \mid e$. Furthermore, for a totally real extension of degree at least $300$, we get by \cite{FDyD} that the root discriminant is at least $50$, and so we can unconditionally replace the fourth condition with $25 \mid e$ in the totally real case. Also note that by Table 2 in \cite{AOT}, once the extension has degree $10^7$ or more, the root discriminant is at least $22.3$, and so we must have that $e \geq 15$ in this scenario since the inequality $19.2 \leq 5^{1 + v_5(e) - \frac{1}{e}}$ becomes $22.3 \leq 5^{1 + v_5(e) - \frac{1}{e}}$.
\end{remark}

\subsubsection{Non-solvable Extensions}
Harbater showed in \cite{DH} that if $G \in \pi_A(U_2)$ and $\lvert G \rvert \leq 300$, then $G$ is solvable. In \cite{JLH}, Hoelscher strengthened this result and proved that if $2 \leq p < 23$ is prime and $G \in \pi_A(U_p)$ with $\lvert G \rvert \leq 300$, then $G$ is solvable. In this section we make further improvements and obtain the following:
\begin{theorem} \label{GLT660}
	If $2 \leq p < 37$ is a prime number and $G \in \pi_A(U_p)$ with $\lvert G \rvert < 660$, then $G$ is solvable.
\end{theorem}

To prove this, we will first extend Hoelscher's result to hold for any prime $p < 37$. We will then systematically rule out the remaining non-solvable groups of order less than $660$ from being elements of $\pi_A(U_p)$ for all $p < 37$.

\begin{example} \label{S5A5PSL27}
	By Theorem 4.1 in \cite{JR1P}, if $p < 37$, then $S_5 \notin \pi_A(U_p)$ and $A_5 \notin \pi_A(U_p)$.
	
	We now show that $\operatorname{PSL}(2,7) \notin \pi_A(U_p)$ for $23 \leq p < 37$. Suppose for contradiction there is an extension $K/\mathbb{Q}$ with $\operatorname{Gal}(K/\mathbb{Q}) \cong \operatorname{PSL}(2,7)$ such that $23 \leq p < 37$ is the only ramified finite prime. This group has order $168 = 2^3 \cdot 3 \cdot 7$. Hence, the ramification in $K/\mathbb{Q}$ must be tame as $p \nmid 168$. Thus, the inertia group for any prime lying over $p$ must be cyclic. $\operatorname{PSL}(2,7)$ has cyclic subgroups of orders $1,2,3,4,$ and $7$. Thus, the corresponding ramification indices satisfy $e \leq 7$ and the root discriminant satisfies
	\[	\lvert \Delta \rvert^{\frac{1}{168}} \leq p^{1 + v_p(e) - \frac{1}{e}}.	\]
	By \cite{FDyD}, we know that $\lvert \Delta \rvert^{\frac{1}{168}} \geq 17.95$ and so
	\[ 17.95 \leq p^{1 + v_p(e) - \frac{1}{e}} \leq p^{1 + 0 - \frac{1}{7}} = p^{\frac{6}{7}}.	\]
	Hence,
	\[	p \geq 17.95^{\frac{7}{6}} > 29,	\]
	and so $p = 31$.
	
	If $e \neq 7$, then $e \leq 4$. But then the root discriminant is at most $31^{1 + 0 - \frac{1}{4}} < 17.95$ which is a contradiction. So, $e = 7$. Because the ramification is tame, the inertia group for any prime embeds into the multiplicative group of the residue field. Letting $f$ denote the residue degree, we have $31^f \equiv 1 \pmod e$. Furthermore, if $r$ is the number of primes that $p$ splits into, we know $ref = 168$ and so $rf = 24$. From $31^f \equiv 1 \pmod e$ and $f \mid 24$, we conclude that $f \in \{6, 12, 24\}$. Note now that $ef$ is equal to the order of the decomposition groups which are subgroups of $\operatorname{PSL}(2,7)$. Since $\operatorname{PSL}(2,7)$ has neither a subgroup of order $42$ nor a subgroup of order $84$, we conclude that $f = 24$ and the decomposition group has order $168$. This means that the decomposition groups are $\operatorname{PSL}(2,7)$. This is a contradiction because any decomposition group must be solvable, whereas $\operatorname{PSL}(2,7)$ is not. 
\end{example}

We can now extend Hoelscher's result to include all primes less than $37$:
\begin{proposition} \label{GLEQ300}
	If $2 \leq p < 37$ is a prime number and $G \in \pi_A(U_p)$ and $\lvert G \rvert \leq 300$, then $G$ is solvable.
\end{proposition}
\begin{proof}
	We already know by the proposition in \cite{JLH} that the above statement holds for $2 \leq p < 23$. An analogous proof now works for $23 \leq p < 37$. That is, suppose for contradiction that $23 \leq p < 37$ and that there is a non-solvable $G \in \pi_A(U_p)$ with $\lvert G \rvert \leq 300$. Let $G$ be such a group with smallest possible order. If $N$ is any nontrivial, normal subgroup of $G$, then $G/N$ is also in $\pi_A(U_p)$. Since $G/N$ has smaller order than $G$, the minimality assumption on $G$ implies that $G/N$ is solvable. Since $G$ itself is not solvable, $N$ cannot be solvable. Thus, $\lvert N \rvert \geq 60$ and so $\lvert G/N \rvert \leq 5$ and $G/N$ is abelian. By Lemma 2.5 in \cite{DH}, $G$ is isomorphic to one of $S_5, A_5$, or $\operatorname{PSL}(2,7)$. This is impossible by Example \ref{S5A5PSL27}, and yields the desired contradiction.
\end{proof}

The following examples examine the remaining possible non-solvable groups of order less than $660$, and demonstrate that none of them appear in $\pi_A(U_p)$ for $p < 37$.

\begin{example} \label{G336}
	After $300$, the next non-solvable groups have order $336$. There are three such groups. Two of them have a normal subgroup isomorphic to $\mathbb{Z}/2\mathbb{Z}$. For each of them, the quotient by this group is a non-solvable group of order $168$. Since Proposition \ref{GLEQ300} says there are no non-solvable groups of order $168$ in $\pi_A(U_p)$ for $p < 37$, neither of these two groups can be in $\pi_A(U_p)$ for $p < 37$.
	
	The third group is isomorphic to $\operatorname{PGL}(2,7)$. Suppose there is a $K/\mathbb{Q}$ which realizes $\operatorname{PGL}(2,7)$ in $\pi_A(U_p)$. $\operatorname{PGL}(2,7)$ has a subgroup of order $42$. The fixed field for this subgroup would yield a non-Galois, degree $8$ extension of $\mathbb{Q}$. Since the normal subgroups in $\operatorname{PGL}(2,7)$ have indices $1,2,$ and $336$, the Galois closure of the degree $8$ sub-extension must be all of $K$. Thus, the largest power of $2$ dividing the Galois closure is $2^4 = 16$. By Corollary 2.3 in \cite{JUA2}, $p \neq 2$. If $p$ were $3$, the root discriminant would be at most $3^{1 + v_3(e) - \frac{1}{e}} \leq 3^{1+1-0} = 9$; but by \cite{FDyD}, the root discriminant is at least $19.47$. By Theorem 4.1 in \cite{LS}, $p \neq 7$.
	
	Primes larger than $7$ do not divide $336$, and so the extension must be tamely ramified. Therefore, the inertia group of any prime is cyclic. The cyclic subgroups of $\operatorname{PGL}(2,7)$ have orders $1, 2, 3, 4, 6, 7,$ and $8$. Thus, the ramification indices satisfy $e \leq 8$. This means the root discriminant is at most $p^{1+v_p(e) - \frac{1}{e}} \leq p^{1+0 - \frac{1}{8}} = p^{\frac{7}{8}}$. Since it is also at least $19.47$, we get $p \geq 19.47^{\frac{8}{7}} > 29$.
	 
	 Lastly, we consider $p = 31$. If $e \leq 7$, then the root discriminant is not large enough; so, $e=8$. The polynomial $x^6 + 2x^5 + 94x^4 + 126x^3 + 2947x^2 + 1736x + 30691$ generates an $S_3$-extension of $\mathbb{Q}$ in which $31$ is the only finite prime that ramifies; it is the Hilbert class field of $\mathbb{Q}(\sqrt{-31})$. Call this extension $H_{\mathbb{Q}(\sqrt{-31})}$. Since $\operatorname{PGL}(2,7)$ has no index $6$ normal subgroup, $K \cap H_{\mathbb{Q}(\sqrt{-31})} \neq H_{\mathbb{Q}(\sqrt{-31})}$. Thus, $K \cap H_{\mathbb{Q}(\sqrt{-31})} = \mathbb{Q}(\sqrt{-31})$. So, \[ [KH_{\mathbb{Q}(\sqrt{-31})}:\mathbb{Q}] = \frac{336 \cdot 6}{2} = 1008. \]
	 $\operatorname{Gal}\left(KH_{\mathbb{Q}(\sqrt{-31})}/\mathbb{Q}\right)$ is a subgroup of $\operatorname{PGL}(2,7) \times S_3$. Since the inertia groups in $\operatorname{PGL}(2,7)$ have order $8$ and the inertia groups in $S_3$ have order $2$, the ramification indices for $KH_{\mathbb{Q}(\sqrt{-31})}/\mathbb{Q}$ are at most $8$. This means that the root discriminant is at most $31^{\frac{7}{8}} < 20.2$. However, by \cite{FDyD}, the root discriminant is at least $20.9$ for degree $1008$ extensions. So, $\operatorname{PGL}(2,7) \notin \pi_A(U_{31})$.
\end{example}

\begin{example} \label{G360}
	The next possible order of a non-solvable group is $360$. There are $6$ such groups. Five of them have a normal subgroup isomorphic to $\mathbb{Z}/3\mathbb{Z}$. In each case, the quotient group is non-solvable of order $120$. But, by Proposition \ref{GLEQ300}, there are no non-solvable groups of order $120$ in $\pi_A(U_p)$ for $2 \leq p < 37$.
	
	The last remaining group is $A_6$. By Theorem 4.2 in \cite{JR1P}, $A_6 \notin \pi_A(U_p)$ for $2 \leq p < 37$.
\end{example}

\begin{example} \label{G420}
	The next candidate non-solvable group has order $420$. The only non-solvable group of order $420$ is $\mathbb{Z}/7\mathbb{Z} \times A_5$. The $\mathbb{Z}/7\mathbb{Z}$ factor forms a normal subgroup, and the quotient yields a non-solvable group of order $60$. But, Proposition \ref{GLEQ300} tells us there is no non-solvable group of order $60$ in $\pi_A(U_p)$ for $2 \leq p < 37$, and so the same is true of the unique non-solvable group of order $420$.
\end{example}

\begin{example} \label{G480}
	There are $26$ non-solvable groups of order $480$. Each of them has a normal subgroup isomorphic to $\mathbb{Z}/2\mathbb{Z}$. In each case, the quotient group is non-solvable of order $240$. Applying Proposition \ref{GLEQ300} now tells us that no such group appears in $\pi_A(U_p)$ for $2 \leq p < 37$.
\end{example}

\begin{example} \label{G504}
	There are two non-solvable groups of order $504$. One has a normal subgroup isomorphic to $\mathbb{Z}/3\mathbb{Z}$. The quotient group is non-solvable of order $168$, and so the group cannot appear in $\pi_A(U_p)$ for $2 \leq p < 37$ by Proposition \ref{GLEQ300}.
	
	The other group is the simple group $\operatorname{PSL}(2,8)$. Suppose $K \in \pi_A(U_p)$ with $\operatorname{Gal}(K/\mathbb{Q}) \cong \operatorname{PSL}(2,8)$. $\operatorname{PSL}(2,8)$ has a subgroup of order $56$, and the fixed field would yield a degree $9$ sub-extension of $\mathbb{Q}$. Because $\operatorname{PSL}(2,8)$ is simple, the Galois closure of this subfield is $K$. Corollary 4.2 in \cite{LS9} now tells us that $p \geq 11$. Since $7$ is the largest prime dividing $504$, the ramification is tame and so the inertia groups are cyclic. The largest size of a cyclic subgroup is $9$, and so $e \leq 9$. The root discriminant is at most $p^{1 + v_p(e) - \frac{1}{e}} \leq p^{\frac{8}{9}}$. By \cite{FDyD}, the root discriminant is at least $20.114$. Thus, $p^{\frac{8}{9}} \geq 20.114$ and so $p >29$.
	
	We now consider $p = 31$. If $ e \neq 9$, then $e \leq 7$ and the root discriminant is not large enough; thus, $e = 9$. Since $ef$ is the order of the decomposition groups, it must also be the order of some subgroup of $\operatorname{PSL}(2,8)$. Examining the possible orders of subgroups of $\operatorname{PSL}(2,8)$, we get that $f \in \{1,2,56\}$. If $f = 56$, the decomposition groups are all of $\operatorname{PSL}(2,8)$; this is impossible since the decomposition groups must be solvable. Thus, $f = 2$ or $f=1$. But, the inertia groups embed into the multiplicative groups of the residue fields, and so $31^f \equiv 1 \pmod e$. That is, $31 \equiv 1 \pmod 9$ or $31^2 \equiv 1 \pmod 9$. This is a contradiction, and so $\operatorname{PSL}(2,8) \notin \pi_A(U_{31})$.
\end{example}

\begin{example} \label{G540}
	There are two non-solvable groups of order $540$. Both have a normal subgroup isomorphic to $\mathbb{Z}/3\mathbb{Z}$. The quotient in both cases is a non-solvable group of order $180$. Proposition \ref{GLEQ300} now rules out either of these groups from appearing in $\pi_A(U_p)$ for $2 \leq p < 37$.
\end{example}

\begin{example} \label{G600}
	There are five non-solvable groups of order $600$. Each has a normal subgroup isomorphic to $\mathbb{Z}/5\mathbb{Z}$. The quotients are non-solvable of order $120$. Proposition \ref{GLEQ300} now shows that none of these groups appear in $\pi_A(U_p)$ for $2 \leq p < 37$.
\end{example}

We conclude by remarking that \ref{GLEQ300}, \ref{G336}, \ref{G360}, \ref{G420}, \ref{G480}, \ref{G504}, \ref{G540}, and \ref{G600}, along with the fact that the next smallest non-solvable group has order $660$, provide a proof for Theorem \ref{GLT660}.

\subsection{Extensions Ramified at Arbitrary Sets of Primes}
We now explore the situation in which more than a single finite prime is allowed to ramify. 
\begin{proposition} \label{relprimegroups} 
	Let $m \in \mathbb{N}_{>1}$ be a natural number and let $n \in \mathbb{N}_{>1}$ be a square-free natural number.
	\begin{enumerate}
		\item  If $\gcd\left(m,\prod_{p \mid n}p-1\right) = 1$, then no groups of order $m$ are in $\pi_A^t(U_n)$.
		\item If $\gcd\left(m,\prod_{p \mid n}p(p-1)\right) = 1$, then no groups of order $m$ are in $\pi_A(U_n)$.
	\end{enumerate}
\end{proposition}
\begin{proof}
	Because of root discriminant bounds, there are no tame extensions of $\mathbb{Q}$ in which $2$ is the only finite prime that ramifies; so, both statements above hold when $n = 2$.
	
	Now let $n$ have at least one odd prime factor. Suppose for contradiction that $G \in \pi_A^t(U_n)$ is a group of order $m$ with $\gcd\left(m,\prod_{p \mid n}p-1\right) = 1$. This means that $m$ is odd, and by Feit-Thompson, $G$ is solvable. So, $G$ has a nontrivial, abelian quotient, $A$. Since $G \in \pi_A^t(U_n)$, we also have that $A \in \pi_A^t(U_n)$. By Kronecker-Weber, $A$ is the Galois group of a sub-extension of $\mathbb{Q}(\zeta_n)/\mathbb{Q}$, and so the order of $A$ divides $\prod_{p \mid n}p-1$. Since also $\lvert A \rvert \mid \lvert G \rvert = m$, this contradicts $\gcd\left(m,\prod_{p \mid n}p-1\right) = 1$.
	
	Suppose now that $G \in \pi_A(U_n)$ is a group of order $m$ satisfying the hypothesis of $2$. Again, $m$ must be odd and an application of Feit-Thompson yields a nontrivial, abelian quotient, $A$. Kronecker-Weber tells us that $A$ is the Galois group of a sub-extension of $\mathbb{Q}(\zeta_{n^{t+1}})/\mathbb{Q}$ for some $t \in \mathbb{N}$. Thus, the order of $A$ divides $\prod_{p \mid n}p^t(p-1)$. Since also $\lvert A \rvert \mid \lvert G \rvert = m$, this contradicts $\gcd\left(m,\prod_{p \mid n}p(p-1)\right) = 1$.
\end{proof}

\begin{example}
	A Fermat prime is a prime number of the form $2^k + 1$ for some $k \in \mathbb{N}$. As a consequence of Proposition \ref{relprimegroups}, no group of odd order can be the Galois group of a tame extension of $\mathbb{Q}$ in which only Fermat primes ramify. 
\end{example}

We now show that if $n_1 \neq n_2$, then $\pi_A(U_{n_1}) \neq \pi_A(U_{n_2})$.
\begin{proposition}
	Let $n \in \mathbb{N}$ be square-free. Then, $\pi_A(U_n)$ determines $n$.
\end{proposition}
\begin{proof}
	We show that $p \mid n$ if and only if  $\forall m \in \mathbb{N}, \mathbb{Z}/p^m\mathbb{Z} \in \pi_A(U_n)$. This shows that $\pi_A(U_n)$ determines the prime factors of $n$, which then determines $n$ since $n$ is square-free.
	
	So suppose $p \mid n$. Then, $\forall m \in \mathbb{N}, \mathbb{Q}(\zeta_{p^{m+1}})$ has a sub-extension $K_m/\mathbb{Q}$ with $\operatorname{Gal}(K_m/\mathbb{Q}) \cong \mathbb{Z}/p^m\mathbb{Z}$. Since $K_m \leq \mathbb{Q}(\zeta_{p^{m+1}})$, the only finite prime that ramifies in $K_m/\mathbb{Q}$ is $p$. Hence, $\mathbb{Z}/p^m\mathbb{Z} \in \pi_A(U_n)$.
	
	Now suppose that $\forall m \in \mathbb{N}, \mathbb{Z}/p^m\mathbb{Z} \in \pi_A(U_n)$. Write the prime factorization of $n$ as $n = q_1\dots q_k$ where each $q_i$ is prime. Our goal is to show that one of the $q_i$ is equal to $p$.  By Kronecker-Weber, each $\mathbb{Z}/p^m\mathbb{Z}$-extension is contained in some cyclotomic field. So, $\forall m \in \mathbb{N}, \exists t_m \in \mathbb{N}$ such that the extension providing witness to the fact that $\mathbb{Z}/p^m\mathbb{Z} \in \pi_A(U_n)$ is a sub-extension of $\mathbb{Q}(\zeta_{t_m})/\mathbb{Q}$. Furthermore, since only primes dividing $n$ may ramify, $t_m$ may be chosen so that its set of prime factors is contained in $\{q_1, \dots q_k\}$. That is, $t_m = q_1^{e_{m,1}} \dots q_k^{e_{m,k}}$ where each $e_{m,i}$ is a nonnegative integer. Notice now that $p^m$ divides $\phi(t_m) = [\mathbb{Q}(\zeta_{t_m}):\mathbb{Q}]$ since $\mathbb{Q}(\zeta_{t_m})/\mathbb{Q}$ has a sub-extension with Galois group isomorphic to $\mathbb{Z}/p^m\mathbb{Z}$ . However, $\phi(t_m) = q_1^{e_{m,1} - 1}(q_1-1) \dots q_k^{e_{m,k} - 1}(q_k-1)$. If $p$ were not equal to one of the $q_i$, the maximal power of $p$ dividing $\phi(t_m)$ would be the maximal power of $p$ dividing $(q_1-1)\dots(q_k-1)$. This expression is independent of $m$, and so we may choose an $m \in \mathbb{N}$ large enough so that $p^m$ does not divide it. Hence, $p$ must equal one of the $q_i$.
\end{proof}

The following shows that if we only consider tame extensions, we can no longer recover $n$ from $\pi_A^t(U_n)$.
\begin{proposition}
	$\pi_A^t(U_6) = \pi_A^t(U_2) \cup \pi_A^t(U_3) = \pi_A^t(U_3)$.
\end{proposition}
\begin{proof}
		First note that $\pi_A^t(U_2)$ consists only of the trivial group and so the final equality holds. Now consider $\pi_A^t(U_6)$. Any extension tamely ramified only at $2$ and $3$ has root discriminant at most $6$. By \cite{FDyD}, the degree of such an extension is at most $9$. So, let $G \in \pi_A^t(U_6)$. We will show it is in $\pi_A^t(U_3)$. Suppose not. Then $2$ must be ramified in the extension that realizes $G$. Since the extension is tame, its degree is not a power of $2$. Since $\pi_A^t(U_2)$ contains only the trivial group, $3$ must also be ramified. Since $3$ is tamely ramified, the degree is not $3$ or $9$. This leaves $5,6, \text{ and }7$ as possibilities for the degree. Any group of order $5$ or $7$ is cyclic, and by Kronecker-Weber we see that neither $\mathbb{Z}/5\mathbb{Z}$ nor $\mathbb{Z}/7\mathbb{Z}$ is in $\pi_A^t(U_6)$. Hence, the degree is $6$ and $G$ is $\mathbb{Z}/6\mathbb{Z}$ or $S_3$. Again by Kronecker-Weber, we can rule out $\mathbb{Z}/6\mathbb{Z}$. Thus, $G$ is $S_3$. $S_3$ has $A_3$ as an index $2$ subgroup. The fixed field for $A_3$ would be a quadratic extension of $\mathbb{Q}$. Since it is tamely ramified only at $2, 3$ and possibly $\infty$, it must be $\mathbb{Q}(\sqrt{-3})$. The $S_3$-extension of $\mathbb{Q}$ is degree $3$ over $\mathbb{Q}(\sqrt{-3})$, so is abelian over it. It is tamely ramified, so only $2$ can be ramified; $\infty$ cannot ramify as $\mathbb{Q}(\sqrt{-3})$ is already totally imaginary. So, the extension is in the ray class field for the modulus $(2)$. However, the ray class number for the modulus $(2)$ for $\mathbb{Q}(\sqrt{-3})$ is $1$, and so there is no degree $3$ abelian extension of it that is ramified only at $2$. This is a contradiction. Thus $\pi_A^t(U_6) = \pi_A^t(U_2) \cup \pi_A^t(U_3) = \pi_A^t(U_3)$. Note also that even the set of number fields tamely ramified only at $2$, $3$, and $\infty$ is just the set of number fields tamely ramified only at $3$ and $\infty$.
\end{proof}

\section{The Minimal Number of Generators of Galois Groups}
Harbater originally stated Conjecture \ref{HConj} in \cite{DH}. In this section we study how the minimal number of generators of a Galois group relates to the ramification in the corresponding extension, and prove the validity of Harbater's conjecture in some special situations.

\begin{remark}
	In the analogous statement in the function field case that inspired Conjecture \ref{HConj}, which originated from \cite[Chapter XIII, Corollary 2.12]{SGA1}, the base of the logarithm is the characteristic. In the arithmetic case, we make the natural choice of $e$ as the base instead. The addition of the constant is because in the function field case we may require one extra generator, the Frobenius, if we do not restrict our attention to regular extensions. Additionally, the analogous statement for curves of higher genus in the function field case would require a constant depending on the genus. If in the analogy between number fields and function fields the ``genus'' of $\mathbb{Q}$ is not $0$, then this would be accounted for by the extra constant in the conjecture.
\end{remark}

\subsection{The Nilpotent Case}
We first consider nilpotent extensions of $\mathbb{Q}$.
\begin{proposition} \label{HCTN}
	If $G \in \pi_A^t(U_n)$ is nilpotent, then $d(G) \leq \log(n)$.
\end{proposition}
\begin{proof}
	Because $G$ is nilpotent, $d(G) = \max \lbrace d(P) | P \text{ is a Sylow subgroup of } G \rbrace$. Since each Sylow subgroup is isomorphic to some quotient of $G$, each Sylow subgroup is also in $\pi_A^t(U_n)$. Thus, we may restrict our attention to the case in which $G$ is a $p$-group.

	Since $G$ is a $p$-group, by the Burnside basis theorem
	\[ G/\phi(G) \cong \left( \mathbb{Z}/p\mathbb{Z} \right)^{d(G)}.	\]
	Hence, $\left( \mathbb{Z}/p\mathbb{Z} \right)^{d(G)} \in \pi_A^t(U_n)$ as well. By the Kronecker-Weber theorem, if $2$ is ramified in an abelian extension then it is wildly ramified. Again by the Kronecker-Weber theorem, each odd prime that ramifies in an abelian extension can increase the minimal size of a generating set of the Galois group by at most $1$. Hence, at least $d\left(\left( \mathbb{Z}/p\mathbb{Z} \right)^{d(G)}\right) = d(G)$ many odd primes must ramify in the extension providing witness to the fact that $\left( \mathbb{Z}/p\mathbb{Z} \right)^{d(G)} \in \pi_A^t(U_n)$. Thus, at least $d(G)$ many odd primes divide $n$ and so $d(G) \leq \log(n)$
\end{proof}

\begin{remark}
	In the nilpotent case we may drop the tameness assumption in Harbater's conjecture as long as we use $C=2$ instead of $C=0$. That is, if $G \in \pi_A(U_n)$ is nilpotent, then $d(G) \leq \log(n) + 2$. The proof would proceed as in Proposition \ref{HCTN}, except now $2$ may ramify. If $2$ ramifies, the Kronecker-Weber theorem tells us that this contributes at most $2$ to the minimal size of a generating set of the Galois group.
\end{remark}

\subsection{Groups with an Index $2$ or Index $3$ Nilpotent Subgroup}
We now prove the validity of Harbater's conjecture if we restrict to Galois groups having an index $2$ or an index $3$ nilpotent subgroup.

\subsubsection{The Index 2 Case}
We start with the index $2$ case by examining nilpotent extensions of quadratic number fields.
\begin{lemma} \label{CNQ}
	There is a constant $C$ such that if $F$ is any quadratic extension of $\mathbb{Q}$ with discriminant $d$ and class number $h$, then $\log_2(h) < C + .8 \cdot \log\left(\frac{\lvert d \rvert}{4}\right)$.
\end{lemma}
\begin{proof}
	 List the quadratic number fields,  $F_1,F_2,\dots,F_i,\dots$, ordered by increasing size of the absolute value of their discriminants, $\lvert d_i \rvert$. Let $h_i$ and $R_i$ denote the class number of $F_i$ and the regulator of $F_i$ respectively. By the Brauer-Siegel theorem in \cite{RB}, for all $\epsilon > 0,$ there is an $N \in \mathbb{N}$ such that if $i > N$, 
	 \[ \frac{\log\left(h_iR_i\right)}{\log\left(\lvert d_i \rvert^{\frac{1}{2}}\right)} < (1+\epsilon). \] Let $\epsilon = .1$. Then for $i > N$ we have 
	\[	\log(h_iR_i) < 1.1 \cdot \log\left(\lvert d_i \rvert^{\frac{1}{2}}\right). 	\] 
	Since there are only finitely many number fields of bounded discriminant, we may also choose $N$ large enough so that if $i>N$, then $\lvert d_i \rvert > 4$. 
	
	Let 
	\[ C_1 = \max\{\log_2(h_i)\}_{1\leq i \leq N} + \left| \log\left(\frac{1}{4}\right) \right|. \]
	Then, for $1 \leq i \leq N$,
	\[ \log_2(h_i) < C_1 + .8 \cdot \log\left( \frac{\lvert d_i \rvert}{4} \right).		\]
	For $i>N$,
	\[ h_i < \frac{1}{R_i} \cdot \lvert d_i \rvert^{\frac{1.1}{2}}.  \]
	Taking the base $2$ logarithm of both sides of this inequality, we get
	\begin{equation*}
		\begin{split}
			\log_2(h_i) & < \log_2\left(\frac{1}{R_i}\right) + \frac{1.1}{2\log(2)} \cdot \log\left(\lvert d_i \rvert\right)  \\
			& < \log_2\left(\frac{1}{R_i}\right) + .8 \cdot \log\left(\lvert d_i \rvert\right) \\
			& = \log_2\left(\frac{1}{R_i}\right) + .8\cdot \log(4)  + .8 \cdot \log\left( \frac{\lvert d_i \rvert}{4} \right).
		\end{split}
	\end{equation*}	
	By \cite{ADF}, the regulator of a quadratic number field is larger than $.48$. Hence, there is a constant $C_2$ such that
	\[ \log_2\left( \frac{1}{R_i} \right) + .8 \cdot \log(4) < C_2.	\]
	Thus,
	\[	\log_2(h_i) < C_2 + .8 \cdot \log\left(\frac{\lvert d_i \rvert}{4}\right). 	\]
	
	\noindent Letting $C = \max\{C_1,C_2\}$ completes the proof of the lemma.
\end{proof}

We now consider the case where the index $2$ subgroup is abelian.
\begin{lemma} \label{ATQ}
	There is a constant $C$ such that if $K/\mathbb{Q}$ is any extension of $\mathbb{Q}$ with a quadratic sub-extension, $F$, over which $K$ is abelian Galois and tamely ramified and if $K/\mathbb{Q}$ is unramified outside of primes dividing $n$ and $\infty$, then $d\left(\operatorname{Gal}(K/F)\right) + 1 \leq \log(n) + C$.
\end{lemma}
\begin{proof}
	Let $K$ be a number field satisfying the hypotheses of the lemma. We will construct a constant, independent of $K$, for which the above inequality holds.
	
	Let $F$ be the quadratic sub-extension of $K/\mathbb{Q}$. Since $K/F$ is abelian, $K$ is a subfield of some ray class field of $F$ for some modulus $\mathfrak{m}$. 
	
	$$\begin{tikzpicture}
	\matrix(m)[matrix of math nodes,
	row sep=3em, column sep=4.5em,
	text height=1.5ex, text depth=0.25ex]
	{ \text{Ray class field for the modulus} \; \mathfrak{m}  \\
		K \\
		F	  \\
		\mathbb{Q} \\ };
	\path[-,font=\scriptsize]
	(m-1-1) edge node[auto] {} (m-2-1)
	(m-2-1) edge node[left] {\text{Abelian}} (m-3-1)
	(m-3-1) edge node[left] {\text{Degree} 2} (m-4-1)
	;
	\end{tikzpicture}.$$
	Since $K/F$ is a tame extension, we may assume that the highest power of each prime ideal dividing $\mathfrak{m}$ is $1$. Furthermore, we may assume that each prime ideal $\mathfrak{P} \mid \mathfrak{m}$ ramifies in $K/F$, for otherwise we can replace $\mathfrak{m}$ with $\frac{\mathfrak{m}}{\mathfrak{P}}$. Let $m$ be the square-free integer obtained by multiplying together all the integral primes lying under some prime ideal $ \mathfrak{P} \mid \mathfrak{m}$. Since each prime ideal $\mathfrak{P} \mid \mathfrak{m}$ ramifies in $K/F$, each prime integer $p \mid m$ ramifies in $K/\mathbb{Q}$.
	
	Let $Cl_\mathfrak{m}(F)$ denote the ray class group for the modulus $\mathfrak{m}$ and let $Cl(F)$ denote the ideal class group of $F$. By Proposition 3.2.3 in \cite{HC}, $Cl(F)$ is isomorphic to $Cl_\mathfrak{m}(F)$ modulo some homomorphic image of $\left( \mathcal{O}_F / \mathfrak{m} \right)^\ast$. Hence,
	\[	d\left( Cl_\mathfrak{m}(F) \right) \leq d \left( \left( \mathcal{O}_F / \mathfrak{m} \right)^\ast \right) +  d \left( Cl(F) \right).	\]
	Letting $h$ be the class number of $F$, we have
	\[	d\left( Cl(F) \right) \leq \log_2(h).	\]
	Write $\mathfrak{m} = \mathfrak{m}_0 \mathfrak{m}_\infty$ where $\mathfrak{m}_0$ denotes the finite part of $\mathfrak{m}$ and $\mathfrak{m}_\infty$ denotes the infinite part of $\mathfrak{m}.$ By the Chinese remainder theorem and the fact that each prime ideal $\mathfrak{P} \mid \mathfrak{m}_0$  only does so to the first power,
	\begin{equation*}
		\begin{split}
			 \left( \mathcal{O}_F / \mathfrak{m} \right)^\ast & =  \left( \mathcal{O}_F / \mathfrak{m}_0 \right)^\ast \times \left( \mathbb{Z}/2\mathbb{Z}	\right)^{\left| \mathfrak{m}_\infty \right|} \\
			 & \cong \prod_{\mathfrak{P} \mid \mathfrak{m}_0} \left( \mathcal{O}_F / \mathfrak{P} \right)^\ast \times \left( \mathbb{Z}/2\mathbb{Z}	\right)^{\left| \mathfrak{m}_\infty \right|}.
		\end{split}
	\end{equation*}
	Since each $\mathfrak{P}$ is a prime ideal, each $\left( \mathcal{O}_F / \mathfrak{P} \right)^\ast$ is isomorphic to the multiplicative group of some finite field and is cyclic. Because a quadratic number field has at most two infinite places, $\left( \mathbb{Z}/2\mathbb{Z}	\right)^{\left| \mathfrak{m}_\infty \right|}$ is the product of at most two cyclic groups. Moreover, each $p \mid m$ can split into at most two prime ideals in $F$, and so the number of prime ideals $\mathfrak{P} \mid \mathfrak{m}_0$ is at most twice the number of prime integers $p \mid m$. Letting $\omega(m)$ denote the number of prime factors of $m$, we obtain that
	\begin{equation*}
		\begin{split}
			d\left( Cl_\mathfrak{m}(F) \right) & \leq d \left( \prod_{\mathfrak{P} \mid \mathfrak{m}_0} \left( \mathcal{O}_F / \mathfrak{P} \right)^\ast \times \left( \mathbb{Z}/2\mathbb{Z}	\right)^{\left| \mathfrak{m}_\infty \right|} \right) +  d \left( Cl(F) \right) \\
			& \leq 2 \cdot \omega(m) + 2 + \log_2(h). 	
		\end{split}
	\end{equation*}
	Because $K$ is a subfield of the ray class field, $\operatorname{Gal}(K/F)$ is a quotient of $Cl_\mathfrak{m}(F)$. Thus,
	\[	d \left( \operatorname{Gal}(K/F)  \right) \leq d \left( Cl_\mathfrak{m}(F) \right),	\]
	and so
	\[	d\left(\operatorname{Gal}(K/F)\right) + 1 \leq 2 \cdot \omega(m) + 2 + \log_2(h) + 1.	\]
	
	Let $\pi(\cdot)$ denote the prime counting function. Note that
	\[	2 \cdot \omega(m) \leq 2 \cdot \pi \left( 3^{20} \right) + .1 \cdot \log(m).	\]
	This is because each prime $p \mid m$ with $p \leq 3^{20}$ contributes $2$ to the left hand side of the above inequality which is canceled out by the $2 \cdot \pi \left( 3^{20} \right)$ on the right hand side. Each prime $p \mid m$ with $p > 3^{20}$ still only contributes $2$ to the left hand side but contributes $.1 \cdot \log(p) > .1 \cdot \log(3^{20}) > 2$ to the right hand side. Letting
	\[ C_1 = 2 \cdot \pi \left( 3^{20}	\right) + 2 \]
	gives
	\[	2 \cdot \omega(m)  + 2 \leq C_1 + .1 \cdot \log(m). 	\]
	Letting $d$ be the discriminant of $F/\mathbb{Q}$, by Lemma \ref{CNQ} there is a constant $C_2$, independent of $F$, such that
	\[	\log_2(h) + 1 < C_2 + .8 \cdot \log \left( \frac{\left| d \right |}{4} \right).	\]
	Hence,
	\[ d\left(\operatorname{Gal}(K/F)\right) + 1 < C_1 + .1 \cdot \log\left( m \right) + C_2 + .8 \cdot \log \left( \frac{\left| d \right |}{4} \right).	\]
	Let $C = C_1 + C_2 + 2$ and $A = \gcd \left(\lvert d \rvert, m\right)$. Note that $C$ is independent of $K$ and that 
	\begin{equation*}
		\begin{split}
			d\left(\operatorname{Gal}(K/F)\right) + 1 & < (C_1 + C_2) + .1 \cdot \log\left( m \right)  + .8 \cdot \log \left( \frac{\left| d \right |}{4} \right) \\
			 & = (C_1 + C_2) + .1 \cdot \log\left( A \right) + .1 \cdot \log \left( \frac{m}{A}	\right) + .8 \cdot \log(A) + .8 \cdot \log \left( \frac{\lvert d \rvert}{4A} \right) \\
			 & < (C_1 + C_2 + 1) + .9 \cdot \log(A) + .9 \cdot \log \left( \frac{\lvert d \rvert}{4A}	\right) + .9 \cdot \log \left( \frac{m}{A} \right) \\
			 & = (C_1 + C_2 + 1) + .9 \cdot \log \left( A \cdot \frac{\lvert d \rvert}{4A} \cdot \frac{m}{A}	\right) \\
			 & < (C_1 + C_2 + 2) + \log \left( A \cdot \frac{\lvert d \rvert}{4A} \cdot \frac{m}{A}	\right) \\
			 & = C + \log \left( \frac{\frac{\lvert d \rvert}{4} \cdot m}{\gcd \left( \lvert d \rvert, m\right)} \right).
		\end{split}
	\end{equation*}
	Noting that the product of the ramified primes in $K/\mathbb{Q}$ is at least $ \frac{\frac{\lvert d \rvert}{4} \cdot m}{\gcd \left( \lvert d \rvert, m\right)}$ completes the proof of the lemma.
\end{proof}

Finally, we now allow $K$ to be any nilpotent, tamely ramified extension over the quadratic sub-extension.
\begin{theorem} \label{NI2}
	There is a constant $C$ such that for every positive square-free integer $n$, if $G \in \pi_A^t(U_n)$ has a nilpotent subgroup of index $2$, then $d(G) \leq \log(n) + C$.
\end{theorem}
\begin{proof}
	Let $C$ be the constant from Lemma \ref{ATQ}. Let $G \in \pi_A^t(U_n)$ have a nilpotent subgroup $H$ with $[G:H] = 2$. Let $K$ be an extension providing witness to the fact that $G \in \pi_A^t(U_n)$. We must show that $d(G) \leq \log(n) + C$. Without loss of generality, we may assume that all primes dividing $n$ ramify in $K/\mathbb{Q}$, since we can otherwise replace $n$ with the product of the ramified primes in $K/\mathbb{Q}$. Let $F$ be the quadratic number field corresponding to the fixed field for $H$. Note now that
	\[ d(H) = \max \{d(P) \mid P \text{ is  a Sylow subgroup of } H \}.	\]
	So, choose some Sylow subgroup $P \leq H$ such that $d(H) = d(P)$. Letting $S$ denote the product of the remaining Sylow subgroups, we get
	\[	H \cong P \times S. 	\]
	Let $E$ denote the fixed field under $S$ of $K/F$. Hence,
	\[\operatorname{Gal}(E/F) \cong P.	\]
	Finally, take the fixed field, $L$, for the Frattini subgroup of $P$. We get $L/F$ is a sub-extension of $E/F$ with $\operatorname{Gal}(L/F) \cong P/\Phi(P).$
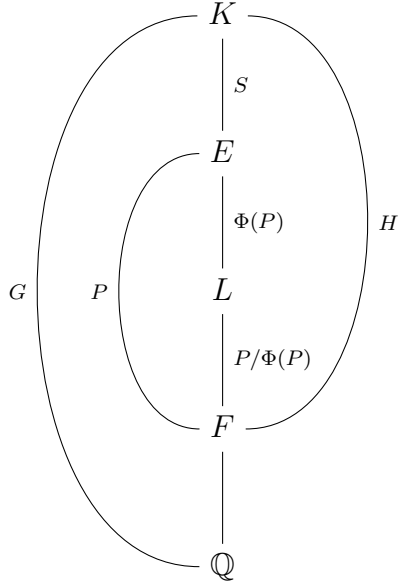
	$$\begin{tikzpicture}
	\matrix(m)[matrix of math nodes,
	row sep=3em, column sep=4.5em,
	text height=1.5ex, text depth=0.25ex]
	{ K	  \\
		E \\
		L	  \\
		F \\
		\mathbb{Q} \\ };
	\path[-,font=\scriptsize]
	(m-1-1) edge node[auto] {$S$} (m-2-1)
	(m-2-1) edge node[auto] {$\Phi(P)$} (m-3-1)
	(m-3-1) edge node[auto] {$P/\Phi(P)$} (m-4-1)
	(m-4-1) edge node[auto] {} (m-5-1)
	(m-1-1) edge[bend right = 90] node[left] {$G$} (m-5-1)
	(m-2-1) edge[bend right = 90] node[left] {$P$} (m-4-1)
	(m-1-1) edge[bend left = 90] node[right] {$H$} (m-4-1)
	;
	\end{tikzpicture}.$$
	By the Burnside basis theorem, $P/\Phi(P)$ is abelian. By Lemma \ref{ATQ},
	\[ d\left( P/\Phi(P)\right) + 1 \leq \log(n) + C.	\]
	Thus,
	\[ d(G) \leq d(H) + 1  = d(P) + 1  = d\left(P/\Phi(P)\right) + 1  \leq \log(n) + C. \]
\end{proof}

\subsubsection{The Index 3 Case}
We now consider the index $3$ situation; the proofs are similar to the index $2$ case. For an integer $d$, we will let $\operatorname{rad}(d) = \prod_{p \mid d, p \text{ prime }} p$. 
\begin{lemma} \label{AUC}
	There is a constant $C$ such that if $F/\mathbb{Q}$ is any cubic extension of $\mathbb{Q}$ with discriminant $d$, then $ d \left( Cl(F) \right) < C + .95 \cdot \log\left(\operatorname{rad}(d) \right)$.
\end{lemma}
\begin{proof}
	Let $F$ be a number field with $[F:\mathbb{Q}] = 3.$ $Cl(F)$ is abelian and so $d \left( Cl(F) \right)$ is equal to the maximal rank of the $p$-Sylow subgroups of $Cl(F)$. We first consider the $2$-Sylow subgroup. By the remark following Theorem 1.1 in \cite{BSTTTZ}, there is a constant $C_1$, independent of $F$, such that the $2$-rank is bounded above by 
	\[ \log_2\left(C_1 \cdot \lvert d \rvert^{.2785}\right) = \log_2(C_1) + \frac{2 \cdot .2785}{\log(2)}\cdot \log\left(\lvert d \rvert^{\frac{1}{2}}\right). \]
	 Letting $C_2 = \log_2(C_1)$ and noting that $\frac{2 \cdot .2785}{\log(2)} < .85$, we get that the $2$-rank is less than
	 \[	C_2 + .85 \cdot \log\left(\lvert d \rvert^{\frac{1}{2}}\right).	\]
	 
	 For the ranks of the other Sylow subgroups we will consider the class group as a whole. An application of the Brauer-Siegel theorem, with $\epsilon = .01$, along with the fact that the regulator is at least $.28$ by \cite{ADF}, allows us to conclude that there is a constant, $C_3$, independent of $F$, such that 
	 \[	\lvert Cl(F) \rvert < C_3 \cdot \lvert d \rvert^{\frac{1+.01}{2}}.	\]
	 Hence, the $p$-rank for $p\geq 3$ is at most
	 \[ \log_3(C_3) + 1.01 \cdot \log_3\left(\lvert d \rvert^\frac{1}{2}\right). \]
	 Letting $C_4 = \log_3(C_3)$ and noting that $\frac{1.01}{\log(3)} < .95$, we get that the $p$-rank is at most
	 \[	C_4 + .95 \cdot \log\left(\lvert d \rvert^\frac{1}{2}\right).	\]
	 Finally, setting $C_5 = \max\{C_2,C_4\}$ gives
	 \[ d \left( Cl(F) \right) < C_5 + .95 \cdot \log\left(\lvert d \rvert^\frac{1}{2}\right).	\]
	 Note that if $p \notin \{2,3\}$, then $v_p(d) \leq 2$.  By \cite[Chapter 3, Section 6]{JPS}, $v_2(d) \leq 3$ and $v_3(d) \leq 5$. So, $\lvert d \rvert^\frac{1}{2} \leq 2 \cdot 3^3 \cdot \operatorname{rad}(d)$. Letting $C = C_5 + \log\left(2 \cdot 3^3\right),$ we conclude that
	 \[ d \left( Cl(F) \right) < C + .95 \cdot \log\left( \operatorname{rad}(d)\right).		\]
\end{proof}

\begin{lemma} \label{ATC}
	There is a constant $C$ such that if $K/\mathbb{Q}$ is any Galois extension of $\mathbb{Q}$ with a cubic sub-extension, $F$, over which $K$ is abelian and tamely ramified and if $K/\mathbb{Q}$ is unramified outside of primes dividing $n$ and $\infty$, then $d\left(\operatorname{Gal}(K/F)\right) + 1 \leq \log(n) + C$.
\end{lemma}
\begin{proof}
	As in Lemma \ref{ATQ} we consider ray class fields. Let $F$ be the cubic sub-extension. Let $\mathfrak{m}$ be the smallest modulus admissible for $K$ and $m$ be the square-free product of integral primes lying under primes dividing $\mathfrak{m}$. We may assume that each $\mathfrak{P} \mid \mathfrak{m}$ only does so to the first power and that each such prime ideal also ramifies in $K/F$. An analogous argument as in Lemma \ref{ATQ} shows that $3 \cdot \omega(m) + 3 + d\left(Cl(F)\right)$ is an upper bound for $d\left(\operatorname{Gal}(K/F)\right)$. Letting $C_1 = 3 \cdot \pi(3^{300})$ gives $3 \cdot \omega(m) \leq C_1 + .01 \cdot \log(m)$. Letting $C_2 = C_1 + 3 + 1$, $C_3$ be the constant from Lemma \ref{AUC}, and $d$ be the discriminant of $F$, we have
	\[	d\left(\operatorname{Gal}(K/F)\right) + 1 < C_2 + .01 \cdot \log(m) + C_3 + .95 \cdot \log\left(\operatorname{rad}(d)\right).	\]
	Let $C = C_2 + C_3 +2$ and $A = \gcd(\operatorname{rad}(d),m)$. Then,
	\begin{equation*}
		\begin{split}
			d\left(\operatorname{Gal}(K/F)\right) + 1 & < (C_2 + C_3) + .01 \cdot \log\left(\frac{m}{A}\right) + .01 \cdot \log(A) + .95 \cdot \log\left(\frac{\operatorname{rad}(d)}{A}\right) + .95 \cdot \log(A) \\
			& < (C_2 + C_3 + 2) + .96 \cdot \log\left(\frac{m}{A}\right) + .96 \cdot \log\left(\frac{\operatorname{rad}(d)}{A}\right) + .96 \cdot \log(A) \\
			& = C + .96 \cdot \log\left(\frac{m \cdot \operatorname{rad}(d)}{A}\right) \\
			& < C + \log \left(\frac{m \cdot \operatorname{rad}(d)}{\gcd(\operatorname{rad}(d),m)} \right).
		\end{split}
	\end{equation*}
	Note now that $\frac{m \cdot \operatorname{rad}(d)}{\gcd(\operatorname{rad}(d),m)}$ is precisely the product of the ramified primes in $K/\mathbb{Q}$, and so is at most $n$. This completes the proof of the lemma.
\end{proof}
\begin{theorem} \label{NI3}
	There is a constant $C$ such that for every positive square-free integer $n$, if $G \in \pi_A^t(U_n)$ has a nilpotent subgroup of index $3$, then $d(G) \leq \log(n) + C$.
\end{theorem}
\begin{proof}
	The proof is identical to Theorem \ref{NI2}, except replace Lemma \ref{ATQ} with Lemma \ref{ATC} and let $[G:H] = 3$ instead of $2$.
\end{proof}

\subsection{Wild Ramification}
\begin{remark}
	Theorem \ref{NI2} and Theorem \ref{NI3} still hold if we expand our attention to extensions of $\mathbb{Q}$ in which primes larger than or equal to $5$ are wildly ramified. Furthermore, if $3$ is unramified in the quadratic or cubic sub-extension of $\mathbb{Q}$, then $3$ may be wildly ramified in the nilpotent extension of the quadratic or cubic. Additionally, the above proofs still hold as written if $2$ or $3$ is wildly ramified in the quadratic or cubic sub-extension of $\mathbb{Q}$.
	
	The only place that tameness was used was in bounding the number of generators of the ray class group by bounding the number of generators of \[ \left(\mathcal{O}_K/\mathfrak{m}\right)^\ast \cong \left(\mathcal{O}_K/\mathfrak{m}_0\right)^\ast \times \left( \mathbb{Z} /2\mathbb{Z} \right)^{\lvert \mathfrak{m}_\infty \rvert}  \cong \prod_{\mathfrak{P} | \mathfrak{m}_0} \left(\mathcal{O}_K/\mathfrak{P}\right)^\ast  \times \left( \mathbb{Z} /2\mathbb{Z} \right)^{\lvert \mathfrak{m}_\infty \rvert}   \] in the proofs of Lemma \ref{ATQ} and Lemma \ref{ATC}. If instead we no longer consider only tame moduli for primes lying above integral primes larger than $3$, or lying above $3$ when $3$ is unramified in the quadratic or cubic, we now get \[ \left(\mathcal{O}_K/\mathfrak{m}\right)^\ast \cong \left(\mathcal{O}_K/\mathfrak{m}_0\right)^\ast \times \left( \mathbb{Z} /2\mathbb{Z} \right)^{\lvert \mathfrak{m}_\infty \rvert}  \cong \prod_{\mathfrak{P} | \mathfrak{m}_0} \left(\mathcal{O}_K/\mathfrak{P}^{k_\mathfrak{P}}\right)^\ast  \times \left( \mathbb{Z} /2\mathbb{Z} \right)^{\lvert \mathfrak{m}_\infty \rvert}   \]
	where $k_\mathfrak{P}$ can be larger than $1$ if it lies over an integral prime larger than $3$ or above $3$ when $3$ is unramified in the quadratic or cubic. By Corollary 4.2.11 in \cite{HC}, since $p \geq \min \{e+2, {k_\mathfrak{P}}\}$ by assumption, we get that 
	\[ \left(\mathcal{O}_K/\mathfrak{P}^{k_\mathfrak{P}}\right)^\ast \cong \left(\mathbb{Z}/(p^f-1)\mathbb{Z}\right) \times \left(\mathbb{Z}/p^q\mathbb{Z}\right)^{(r+1)f} \times \left(\mathbb{Z}/p^{q-1}\mathbb{Z}\right)^{(e-r-1)f}	\]
	where $k_\mathfrak{P} + e -2 = eq+r, 0 \leq r < e$. Note for a quadratic extension that $(r+1)f \leq ef \leq 2,$ and$ (e-r-1)f \leq (e-1)f \leq ef \leq 2$, and for a cubic extension that $(r+1)f \leq ef \leq 3,$ and$ (e-r-1)f \leq (e-1)f \leq ef \leq 3$. So, $\left(\mathcal{O}_K/\mathfrak{P}^{k_\mathfrak{P}}\right)^\ast$ is a product of at most $5$ cyclic groups in the quadratic case, and $7$ cyclic groups in the cubic case. If we still let $m$ be the product of the integral primes lying under those dividing the modulus, adjusting the proof of Lemma \ref{ATQ} for the current situation, we now have 
	\[	d\left(\operatorname{Gal}\left(K/F\right)\right) + 1 \leq 5\cdot \left( 2\cdot \omega(m) \right) + 2 + \log_2(h) +1\]
	instead of
	\[		d\left(\operatorname{Gal}\left(K/F\right)\right) + 1 \leq 2\cdot \omega(m) + 2 + \log_2(h) +1\]
	If we let $C_1 = 10 \cdot \pi(3^{100}) + 2$ instead of $2 \cdot \pi(3^{20}) + 2$, we get 
	\[ 10 \cdot \omega(m) + 2 \leq C_1 + .1 \cdot \log(m) \]
	and the rest of the proof is the same. Adjusting the proof of Lemma \ref{ATC} for the current situation, we get that $7\cdot \left( 3 \cdot \omega(m) \right) + 3 + d\left(Cl(F)\right) + 1$ is an upper bound for $d\left(\operatorname{Gal}(K/F)\right) + 1$. Now let $C_1 = 21 \cdot \pi(3^{2100})$ instead of $3 \cdot \pi(3^{300})$ and the rest of the proof is the same. The proofs of Theorem \ref{NI2} and Theorem \ref{NI3} still work even in this new situation.
\end{remark}

\section{Consequences and Examples}
\begin{proposition} \label{HCtopfingen}
	If Harbater's conjecture holds, then for all $n \in \mathbb{N}, \pi_1^t(U_n)$ is topologically finitely generated.
\end{proposition}
\begin{proof}
	By assumption of Harbater's conjecture, every group in the inverse system whose limit is $\pi_1^t(U_n)$ is generated by at most $C + \log(n)$ elements. Now apply Lemma 2.5.3 in \cite{LRPZ}.
\end{proof}

\begin{remark}
	In light of Proposition \ref{HCtopfingen}, the veracity of Conjecture \ref{HConj} in the cases described in Proposition \ref{HCTN}, Theorem \ref{NI2}, and Theorem \ref{NI3} provides evidence that for all $n \in \mathbb{N}, \pi_1^t(U_n)$ is topologically finitely generated.
\end{remark}

\begin{proposition}
	If Harbater's conjecture holds with a constant $C$, then for any tame extension $K/\mathbb{Q}$, if $m$ is the product of the ramified primes then the class group of $K$ has a generating set of size at most $1 + [K:\mathbb{Q}]\left(\log(m) + C -1 \right)$.
\end{proposition}
\begin{proof}
	Any group $G$ that can be generated by $d$ elements is a quotient of the free group on $d$ elements, $F_d$. So, $G \cong F_d/N$ for some $N \trianglelefteq F_d$. By the correspondence theorem, any subgroup of $G$ is of the form $H/N$ for some $N \leq H \leq F_d$.  Also, $[F_d:H] = [G:H/N]$. Let this index be $n$. By the Nielsen-Schreier theorem, we know that $H$ is free of rank $1+n(d-1)$. So, $H/N$ can be generated by $1+n(d-1)$ elements.
	
	Let $K$ be a tame extension of $\mathbb{Q}$ and let $m$ be the product of the ramified primes in $K/\mathbb{Q}$. Let $n = [K:\mathbb{Q}]$, $H_K$ be the hilbert class field of $K$, and $M$ be the Galois closure of $H_K$ over $\mathbb{Q}$. 
	$$\begin{tikzpicture}
	\matrix(m)[matrix of math nodes,
	row sep=3em, column sep=4.5em,
	text height=1.5ex, text depth=0.25ex]
	{ M	  \\
		H_K \\
		K	  \\
		\mathbb{Q} \\ };
	\path[-,font=\scriptsize]
	(m-1-1) edge node[auto] {} (m-2-1)
	(m-2-1) edge node[auto] {} (m-3-1)
	(m-3-1) edge node[auto] {n} (m-4-1)
	
	;
	\end{tikzpicture}.$$
	
	By assumption of Harbater's conjecture, we get that $\operatorname{Gal}(M/\mathbb{Q})$ requires at most $\log(m) + C$ generators. Since $\operatorname{Gal}(M/K)$ is an index $n$ subgroup, it requires at most $1 + n(\log(m) + C -1)$ generators. Since $\operatorname{Gal}(H_K/K)$, which is isomorphic to the class group of $K$, is a quotient of $\operatorname{Gal}(M/K)$, it also requires at most $1 + n(\log(m) + C -1) = 1 + [K:\mathbb{Q}](\log(m) + C -1)$ generators.
\end{proof}

\begin{example}
	As a consequence of Theorem \ref{NI2} and Theorem \ref{NI3}, if $p$ is a prime number and $n$ is a square-free natural number, then there are only finitely many groups of the form $\left(\mathbb{Z}/p\mathbb{Z}\right)^i \rtimes \mathbb{Z}/2\mathbb{Z}$ or $\left(\mathbb{Z}/p\mathbb{Z}\right)^i \rtimes \mathbb{Z}/3\mathbb{Z}$ in $\pi_A^t(U_n)$. If we also assume that $n$ is coprime to $2$ and $3$, then the same is true for $\pi_A(U_n)$.
	
	The Boston-Markin conjecture in \cite{BM} states that every nontrivial finite group $G$ can be realized as a Galois group over $\mathbb{Q}$ with $\max \left\lbrace 1, d\left(G^{ab}\right) \right\rbrace$ many ramified primes. The above semidirect products are of interest because they are potential cases in which Harbater's conjecture could have clashed with the Boston-Markin conjecture. Theorem \ref{NI2} applies to all generalized dihedral groups, of which elementary abelian $p$-groups semidirect $\mathbb{Z}/2\mathbb{Z}$ by inversion for $p \geq 3$ are a special case. These groups have $\mathbb{Z}/2\mathbb{Z}$ abelianization, so the Boston-Markin conjecture predicts there should be extensions ramified at a single prime that realize each of them as Galois groups over $\mathbb{Q}$. The groups themselves also require as many generators as the rank of the elementary abelian $p$-group, so Harbater's conjecture suggests that the product of the primes in the extensions realizing them would have to be quite large.
\end{example}

\begin{remark}
	The arguments of Section 3.2 actually show that for the corresponding extensions, $d(G) < C + .97 \cdot \log(n)$ where $n$ is the product of the ramified primes. This means that when $n$ is large, $d(G) < \log(n)$ without the aid of the constant. Since each prime can only divide the discriminant a bounded number of times for quadratic and cubic extensions, this means that if the discriminant is large, then $n$ is also large. Since there are only finitely many number fields of bounded discriminant, $n$ is small for only finitely many such extensions and so the constant is necessary for only finitely many such extensions. This provides evidence that the constant should be small, and perhaps even $0$.
\end{remark}

\section*{Acknowledgements}
I would like to thank David Harbater for many helpful discussions and suggestions about the content of this paper.

\bibliographystyle{abbrv}
\bibliography{bib}

{
	\footnotesize
	
	\textsc{Department of Mathematics, University of Pennsylvania, Philadelphia, PA 19104-6395} 
	
	\textit{E-mail address}: \texttt{pollakb@sas.upenn.edu}
}

\end{document}